\documentclass[11pt]{amsart}

\usepackage{enumerate, amsmath, amsthm, amsfonts, amssymb, %xy,  
mathrsfs, graphicx, paralist, ulem}
\usepackage[usenames, dvipsnames]{color}
\usepackage[margin=1in]{geometry} 
\usepackage{comment}

\usepackage{mathabx} %This is for \sqbullet
\usepackage{breqn} %This is for dmath
\usepackage{tikz-cd}
\numberwithin{equation}{section}
\newtheorem{Theorem}[equation]{Theorem}
\newtheorem{Proposition}[equation]{Proposition}
\newtheorem{Hypothesis}[equation]{Hypothesis}
\newtheorem{Lemma}[equation]{Lemma}
\newtheorem{Corollary}[equation]{Corollary}
\newtheorem{Conjecture}[equation]{Conjecture}
\newtheorem{Question}[equation]{Question}

\theoremstyle{definition}
\newtheorem{Remark}[equation]{Remark}
\newtheorem{Caution}[equation]{Caution}

\newtheorem{eg}[equation]{Example}

\newtheorem{Definition}[equation]{Definition}

\renewcommand{\comment}[1]{}

\newcommand{\betacheck}{\beta^\vee}

\newcommand{\bB}{\mathbf{B}}

\newcommand{\bG}{\mathbf{G}}

\newcommand{\cO}{\mathcal{O}}

\newcommand{\cS}{\mathcal{S}}

\newcommand{\cT}{\mathcal{T}}

\newcommand{\RR}{\mathbb{R}}

\newcommand{\ZZ}{\mathbb{Z}}

\renewcommand{\phi}{\varphi}
\renewcommand{\emptyset}{\varnothing}
\newcommand{\eps}{\varepsilon}

\renewcommand{\tilde}[1]{\widetilde{#1}}

\makeatletter
\def\Ddots{\mathinner{\mkern1mu\raise\p@
\vbox{\kern7\p@\hbox{.}}\mkern2mu
\raise4\p@\hbox{.}\mkern2mu\raise7\p@\hbox{.}\mkern1mu}}
\makeatother

\DeclareMathOperator{\im}{image}

\DeclareMathOperator{\sgn}{sgn}

\newcommand{\suchthat}{\mid}

\newcommand{\height}{\mathrm{ht}}
\newcommand{\Inv}{\mathrm{Inv}}

\newcommand{\WTits}{W_{\TitsCone}}
\newcommand{\TitsCone}{\cT}
\newcommand{\textif}{\text{ if }}
\newcommand{\textand}{\text{ }\mathrm{and }\text{ }}
\newcommand{\textor}{\text{ or }}
\newcommand{\Invplusplusx}{\Inv^{++}_x}
\newcommand{\tr}{\tilde{r}} 
\newcommand{\tn}{\tilde{n}}

\renewcommand{\emph}[1]{{\it #1}}
\renewcommand{\epsilon}{\eps}

\begin{document}

\title[On the double-affine Bruhat order]{On the double-affine Bruhat order: the $\epsilon=1$ conjecture and classification of covers in ADE type}
\author{Dinakar Muthiah}
\address{D.M.:
Department of Mathematical and Statistical Sciences,                 
University of Alberta, 
CAB 632, 
Edmonton, AB T6G 2G1,
Canada
}
\email{muthiah@ualberta.ca}

\author{Daniel Orr}
\address{D.O.:
Department of Mathematics, MC 0123,
460 McBryde Hall, Virginia Tech,
225 Stanger St., Blacksburg, VA 24061 USA}
\email{dorr@vt.edu}

%\date{\today}
%\tableofcontents

\begin{abstract}
For any Kac-Moody group $\bG$, we prove that the Bruhat order on the semidirect product of the Weyl group and the Tits cone for $\bG$ is strictly compatible with a $\ZZ$-valued length function. We conjecture in general and prove for $\bG$ of affine ADE type that the Bruhat order is graded by this length function. We also formulate and discuss conjectures relating the length function to intersections of ``double-affine Schubert varieties.''
\end{abstract}

\maketitle

\section{Introduction}

\subsection{The Bruhat order and previous work}

Let $\bG$ be a Kac-Moody group, let $W$ be its Weyl group, and let $\cT$ be the Tits cone of integral coweights. We can form the semi-direct product $\WTits = \cT \rtimes W$, which will in general be a semi-group. In \cite{BKP}, Braverman, Kazhdan, and Patnaik consider the case when $\bG$ is untwisted affine type, and they construct the Iwahori-Hecke algebra for the group $G = \bG(F)$ where $F$ is a non-archimedian local field. A key property of their construction is that they need to restrict attention to functions supported on a subsemigroup (the Cartan semigroup) $G^+ \subsetneq G$. Then they show that the Iwahori-double cosets on $G^+$ are exactly in bijection with the semi-group $\WTits$. 

Additionally, they define a preorder on $\WTits$ that we call the {\bf (double-affine) Bruhat order}; they conjecture that this order is in fact a partial order. In \cite{M}, the first-named author constructs a function
\begin{align}
  \label{eq:22}
  \ell_{\eps} : \WTits \rightarrow \ZZ \oplus \ZZ \eps
\end{align}
that is strictly compatible with the preorder, where $\ZZ \oplus \ZZ \eps$ is ordered lexicographically. As a corollary, he proves that the double-affine Bruhat order is partial order.
However, many questions remained open. In particular, because the intervals in $\ZZ \oplus \ZZ \eps$  are infinite in general, the results of~\cite{M} do not give strong finiteness results about the order $\WTits$.

\subsection{Setting $\eps=1$}

In \cite[Question 5.10]{M}, the question is asked whether the composed function (the {\bf length function})
\begin{align}
  \label{eq:23}
  \ell : \WTits \rightarrow \ZZ \oplus \ZZ \eps \rightarrow \ZZ,
\end{align}
obtained by setting $\epsilon=1$, is strictly compatible with the order on $\WTits$.

Let us briefly describe the ``single-affine'' situation (see Section \ref{subsec:terminology} below for the terminology), when $\bG$ is finite-type and simply-connected. In this case, $\WTits$ coincides with the usual notion of affine Weyl group, which by our simply-connected assumption is a Coxeter group. The order defined by Braverman, Kazhdan, and Patnaik exactly recovers the usual affine Bruhat order on this Coxeter group. Then, as explained in \cite{M}, the function $\ell: \WTits \rightarrow \ZZ$ exactly recovers the usual Coxeter length function; in particular, it is strictly compatible with the order.

In this paper, we give a positive answer to the above question. 
\begin{Theorem}[Theorem \ref{thm:compatibility-of-ell-with-order} below]
  For any Kac-Moody group $\bG$, the length function $\ell$ is strictly compatible with the Bruhat order on $\WTits$. 
\end{Theorem}

Even more interesting than this positive answer is our method of proof and the role that inversion sets play. To explain this, we briefly recall some of the ingredients for the definition of the Bruhat order on $\WTits$. Let $x \in \WTits$. To every positive real root $\beta$ for $\bG$ and every $n \in \ZZ$, there is an associated reflection element that we call $s_{\beta[n]}$. In general, $s_{\beta[n]} \notin \WTits$, but in the cases of interest, we will have $ x s_{\beta[n]} \in \WTits$. In this case, by definition, $x$ and $ x s_{\beta[n]}$ are comparable in the Bruhat order, and the order is generated by such relations. 

First we construct a set $\Inv^{++}_x(s_{\beta[n]})$, which is a certain subset of the inversion set of $s_{\beta[n]}$ ($\Inv^{++}_x(s_{\beta[n]})$ is defined in Section \ref{sec:length-functions-and-the-conjecture}) and we prove that $ x s_{\beta[n]} > x$ if and only if $\Inv^{++}_x(s_{\beta[n]}) \neq \emptyset$. Then, to prove our theorem we show that when $x s_{\beta[n]} > x$:
\begin{align}
  \label{eq:24}
 \ell(x s_{\beta[n]}) - \ell(x) = \# \Inv^{++}_x(s_{\beta[n]}). 
\end{align}
There is a similar statement when the inequality is reversed. To prove \eqref{eq:24}, we develop a  generalized notion of inversion set and we relate $\ell(x)$ with the inversion set of $x^{-1}$ (denoted $\Inv(x^{-1})$) and $\ell(x s_{\beta[n]})$ with the inversion set of $s_{\beta[n]}x^{-1}$ ($\Inv(x^{-1} s_{\beta[n]}$). However, this relationship is very subtle because these inversion sets are generally infinite and the function $\ell$ may take negative values. A particular manifestation of the subtlety of this relationship is that, unlike for Weyl groups, elements of different lengths may have identical inversion sets. Additionally, implicit in \eqref{eq:24} is the finiteness of $\Inv^{++}_x(s_{\beta[n]})$, which is not at all obvious. 

What we show is that there is a canonically defined injection from $\Inv(x^{-1})$ to $\Inv(x^{-1} s_{\beta[n]})$, and we then construct a bijection between $\Inv^{++}_x(s_{\beta[n]})$ and the complement of the image of this injection. We show that $\ell(x)$ can be computed by performing a weighted sum over certain finite subsets of $\Inv(x^{-1})$; similarly for $\ell(xs_{\beta[n]})$. Finally, we need an analogue of the finite-type fact that $2\rho$ is the sum of positive roots. Putting these various ingredients together we get \eqref{eq:24}.

\subsection{Classifying covers}

As a consequence, given $x,y \in \WTits$ with $x < y$, we know that a chain between $x$ and $y$ can have at most $\ell(y) - \ell(x)$ elements. A natural question is whether this bound is always acheived. Equivalently, we can ask whether covers are classified by the function $\ell$. Let us write $x \lhd y$ to denote that $y$ covers $x$ in the Bruhat order. Then we conjecture the following.

\begin{Conjecture}[Conjecture \ref{conj:covers}  below]
{\label{conj:covering-conjecture}}
 Let $x,y \in \WTits$, then $x \lhd y$ if and only if $x < y$ and $\ell(x) = \ell(y) - 1$.
\end{Conjecture}

We note that in the finite and single-affine situations this conjecture is true because it is true for all Coxeter groups. In our double-affine situation those methods are not available. However, in untwisted affine type ADE, we have a positive result.

\begin{Theorem}[Theorem \ref{thm:ade-covering} below]
  Conjecture \ref{conj:covering-conjecture} is true for $\bG$ of untwisted affine ADE type.
\end{Theorem}

In untwisted affine ADE, using our explicit control of the pairings between roots and coroots, we reduce the problem to a calculation that is essentially the case of affine $SL_2$. Then we verify the theorem in this situation by explicit computations.

\subsection{Some remarks on further directions}

Let $\bG$ be an untwisted affine Kac-Moody group, and let $F = k((\pi))$ be the field of formal Laurent series over a field $k$. Let us take $G = \bG(F)$ , and let $G^+ \subsetneq G$ be the Cartan semigroup as in \cite{BKP}. Let $I \subset G$ be the Iwahori subgroup. Then we expect the quotient $G^+/I$ to be the $k$-points of the ``double-affine flag variety''. We expect the $I$-orbits on $G^+/I$ to be the ``double-affine Schubert cells''. By \cite{BKP}, we know that these $I$ orbits are in bijection with $\WTits$. So we expect the closure order on ``double-affine Schubert cells'' to be precisely the double-affine Bruhat order. Unfortunately, in this double-affine situation we do not know how to properly work with $G^+/I$ as object of algebraic geometry. So the statement about the closure order is currently only a heuristic (or perhaps a definition). Of course, in the single-affine case when $\bG$ is finite-type, all of the above has precise meaning and is well known.

In Section \ref{sec:further-questions} we discuss some conjectures motivated by the above heuristics as well as directions for further work. In particular, we write down the sets we expect to be the transverse slices to one double-affine Schubert variety embedded in another. Motivated by these transverse slices, we give a purely group-theoretic but conjectural definition of the double-affine Bruhat order.
We also discuss some further questions that are more combinatorial in nature.

\subsection{A remark about terminology}
\label{subsec:terminology}

We write $k$ to refer to fields we consider without any valuation, and we write $F$ to refer to a field with integer valuation and residue field $k$. We propose the following terminology.

We will use the adjective ``finite'' to refer to finite-type Kac-Moody groups over fields $k$ and objects constructed from them. For example: \emph{finite} roots, \emph{finite} Weyl groups, \emph{finite} flag varieties, \emph{finite} Hecke algebras, \emph{finite} Bruhat orders, etc.

We will use the the adjective ``single-affine'' to refer to affine-type Kac-Moody groups over fields $k$ or finite-type Kac-Moody groups over valued fields $F$ and the objects constructed from them. For example: \emph{single-affine} roots, \emph{single-affine} Weyl groups, \emph{single-affine} flag varieties, \emph{single-affine} Hecke algebras, \emph{single-affine} Bruhat orders, etc.

Finally, we will use the adjective ``double-affine'' to refer to affine-type Kac-Moody groups over valued fields $F$ and the objects constructed from them. For example: \emph{double-affine} roots, \emph{double-affine} Weyl groups, \emph{double-affine} flag varieties, \emph{double-affine} Hecke algebras, \emph{double-affine} Bruhat orders, etc.

% The benefit of this terminology is that currently the adjective ``affine'' is sometimes used ambiguously refer to either the single-affine or double-affine situations. Our proposal would remove that ambiguity when discussing both situations, as we will now. One minor flaw of this terminology is that the cases of non-affine infinite-type Kac Moody groups over $k$ (resp. $F$) are classified under the heading of ``single-affine'' (resp. ``double-affine''). However, because most of the current interest in infinite-type Kac-Moody groups is in affine type, we believe the benefits of the terminology outweigh the disadvantages currently. 

 \subsection{Acknowledgements}
 We would like to thank Manish Patnaik for many discussions and for carefully reading an early version of this manuscript. D.M. was supported by a PIMS postdoctoral fellowship. D.O. was supported by NSF grant DMS-1600653.
  
%   - Others?

\section{Notation}
\label{sec:notation}

\newcommand{\tDelta}{\tilde{\Delta}}

Let $\bG$ be a Kac-Moody group, $W$ its Weyl group, and $\Delta$ the set of real roots of $\bG$.
Define set of \emph{$\bG$-affine roots} as
\begin{align}
  \tDelta = \{\beta + n \pi \in \ZZ\Delta\oplus\ZZ\pi \suchthat \beta\in\Delta, n\in \ZZ\}.
\end{align}
(When $\bG$ is an affine Kac-Moody group, we will use the terminology ``double-affine'' synonymously with $\bG$-affine.) 

Let $\Delta^+\subset\Delta$ be the set of positive real roots of $\bG$. We call $\beta+n\pi\in\tDelta$ positive (and write $\beta+n\pi>0$) provided that $\beta\in\Delta^+$ and $n\ge 0$ or $\beta\in -\Delta^+$ and $n>0$; otherwise we call $\beta+n\pi$ negative (and write $\beta+n\pi<0$). Let $\tDelta^+$ be the subset of positive elements of $\tDelta$.

For any $\beta\in\Delta^+$ and $n\in\ZZ$, let us define
\begin{align}
  \beta[n] = \sgn(n)\cdot (\beta + n \pi) = \sgn(n)\beta + |n| \pi,
\end{align}
where $\sgn : \ZZ \rightarrow \{\pm 1\}$ is the signum function:
\begin{align}
  \label{eq:020}
  \sgn(n) =
  \begin{cases}
    + 1 & \textif n \geq 0 \\
    - 1 & \textif n < 0. 
  \end{cases}
\end{align}
Notice that $\beta[n]$ always belongs to $\tDelta^+$. Conversely, every element of $\tDelta^+$ is of this form for unique $\beta\in\Delta^+$ and $n\in\ZZ$. We refer to $\beta$ as the \emph{$\bG$-classical part} of $\beta[n]$. (When $\bG$ is an affine Kac-Moody group, we will use the terminology ``single-affine'' synonymously with $\bG$-classical.)

Let $P$ the \emph{coweight} lattice of $\bG$ and consider the semidirect product group $W_P=P\rtimes W$. We denote elements of $W_P$ by $\pi^\mu w$ where $\mu\in P$ and $w\in W$. The group $W_P$ acts on the set $\tDelta$ via the formula:
\begin{align}
  \pi^\mu w(\beta+n\pi) = w(\beta)+(n+\langle \mu, w(\beta)\rangle)\pi
\end{align}
where $\langle\cdot,\cdot\rangle$ is the canonical pairing between the coweight lattice and the root lattice of $\bG$. We define the reflection $s_{\beta[n]}$ corresponding to $\beta[n]\in\tDelta^+$ as the following element of $W_{P}$:
\begin{align}
  s_{\beta[n]} = \pi^{n \betacheck} s_\beta
\end{align}
where $\betacheck$ is the real coroot of $\bG$ associated with $\beta$.
The action of $s_{\beta[n]}$ on affine roots is then given by the usual formula
\begin{align}
  s_{\beta[n]}(\gamma[m]) = \gamma[m]- \sgn(m)\sgn(n)\langle\betacheck,\gamma\rangle\beta[n].
\end{align}

For any $x\in W_P$, we define the inversion set $\Inv(x)$ to be the subset of elements of $\tDelta^+$ made negative under the action of $x$. If $x \in W$, then this definition coincides with the usual definition of inversion set (via the injection $\Delta^+\hookrightarrow\tDelta^+,\beta\mapsto\beta[0]$).

Let $\TitsCone\subset P$ be the Tits cone in the coweight lattice of $\bG$. Our main object of study is $\WTits=\TitsCone\rtimes W$, which is a subsemigroup of $W_P$.

\section{The length functions $\ell_\epsilon$ and $\ell$}

\subsection{A height formula}

Let $\rho$ be the sum of fundamental weights of $\bG$ (which we have chosen once and for all). For any $\beta\in\Delta^+$, the height of the associated coroot $\betacheck$ is $\height(\betacheck) = \langle \betacheck,\rho \rangle$, which is independent of the choice of $\rho$.

For any $\beta\in\Delta^+$, we define $|s_\beta| : \Delta^+ \rightarrow \Delta^+$ via:
\begin{align}
  \label{eq:040}
  |s_\beta| (\gamma) = | s_\beta(\gamma) | = 
  \begin{cases}
    s_\beta(\gamma) & \textif s_\beta(\gamma) > 0 \\
    -s_\beta(\gamma) & \textif s_\beta(\gamma) < 0. 
  \end{cases}
\end{align}

\begin{Proposition}
  \label{prop-the-height-formula}
  Let $\beta\in\Delta^+$ and suppose that $\cS \subset \Delta^+$ is a finite subset containing $\Inv(s_\beta)$ and closed under $|s_\beta|$. Then
  \begin{align}
    2 \cdot \height(\betacheck) = \sum_{\gamma \in \cS} \langle \betacheck, \gamma \rangle
  \end{align}
\end{Proposition}

\begin{proof}
Let us consider those $\gamma \in \cS$ such that $\gamma \notin \Inv(s_\beta)$. For such $\gamma$, $|s_\beta(\gamma)| = s_\beta(\gamma)$ and $\langle \beta, |s_\beta|(\gamma)\rangle = - \langle \beta, \gamma \rangle$. So the $\gamma$-term cancels the $|s_\beta(\gamma)|$-term in the above sum whenever $\gamma \notin \Inv(s_\beta)$. Therefore:
\begin{align*}
\sum_{\gamma \in \cS} \langle  \betacheck, \gamma \rangle = \sum_{\gamma \in \Inv(s_\beta)} \langle  \betacheck, \gamma \rangle = \langle \betacheck,  \sum_{\gamma \in \Inv(s_\beta)} \gamma \rangle.
\end{align*}
 
By the definition of $\rho$, we can verify that (see, e.g., \cite[Exercise 3.12]{K}):
\begin{align}\label{E:rho-wrho}
\sum_{\gamma \in \Inv(w)} \gamma = \rho - w^{-1}(\rho)
\end{align}
for any $w\in W$.
So we have:
\begin{align*}
\sum_{\gamma \in \cS} \langle  \betacheck, \gamma \rangle = \langle \betacheck,  \rho - s_\beta(\rho) \rangle = \langle  \betacheck, \rho \rangle + \langle  \betacheck, \rho \rangle = 2 \cdot \height(\betacheck).
\end{align*}
\end{proof}

\newcommand{\floor}[1]{\left\lfloor #1 \right\rfloor}

More generally, for any $\mu\in P$, we define $2\cdot\height(\mu)=\floor{\langle \mu, 2\rho\rangle}$. Here $\floor{\cdot}$ denotes the ``floor'' function, i.e., for any $x \in \RR$, $\floor{x}$ is the unique integer such that $0 \leq x - \floor{x} < 1$.

%\begin{Remark}
We apply the floor function for the psychological benefit of allowing the length function defined in the next section to take on only integral values. However, we will see that the invariant quantities we construct, i.e., independent of choice of $\rho$, will be the {\it differences} of lengths of elements that are comparable in the Bruhat order. These quantities will always be integral even if we allow the length to take on rational values. We note that in untwisted affine cases, $\rho$ can be chosen so that $\langle \mu, 2\rho\rangle \in \ZZ$ for all $\mu \in P$. For simplicity of notation, we shall assume below that $\langle \mu, 2\rho\rangle \in \ZZ$ and omit the $\floor{\cdot}$, which can easily be added to the arguments below. 
%\end{Remark}

% . where we make the following assumption
% \begin{align}
% \label{E:assumption-mu-2rho}
% \langle\mu,2\rho\rangle \in\ZZ,\qquad \forall \mu\in P.
% \end{align}
% It is known that \eqref{E:assumption-mu-2rho} holds for $\bG$ of finite type or simply-connected affine type. 

% {\tt\color{red} *** Dan: check/extend this remark ***}

\subsection{The length functions $\ell_\eps$}
\label{sec:length-functions-and-the-conjecture}

Let us recall the definition of the length function $\ell_\epsilon : \WTits \rightarrow \ZZ \oplus \ZZ \epsilon$ from \cite{M}. For $\pi^\mu w \in \WTits$, we define
\begin{align}
  \ell_\epsilon (\pi^\mu w) = 2 \cdot \height(\mu_+) + \epsilon \cdot \left( \# \{ \gamma \in \Inv(w^{-1}) \suchthat \langle \mu,\gamma \rangle \geq 0 \} - \# \{ \gamma \in \Inv(w^{-1}) \suchthat \langle \mu,\gamma \rangle < 0 \} \right)
\end{align}
where $\mu_+$ is the unique dominant translate of $\mu$ under the action $W$.

Let $\beta[n]\in\tDelta^+$ be a $\bG$-affine root such that $\pi^\mu w s_{\beta[n]} \in \WTits$. If $\pi^\mu w(\beta[n]) > 0$, then we declare that $\pi^\mu w s_{\beta[n]} > \pi^\mu w$ in the Bruhat preorder. In \cite{M}, it is proved that $\pi^\mu w(\beta[n]) > 0$ is equivalent to $\ell_\epsilon(\pi^\mu w s_{\beta[n]}) > \ell_\epsilon(\pi^\mu w) $, where $\ZZ \oplus \ZZ \epsilon$ is ordered lexicographically. So the map $\ell_\epsilon$ is strictly compatible with the order structure. In particular, this implies that the Bruhat preorder is in fact a partial order. However, because intervals in $\ZZ \oplus \ZZ \epsilon$ are not finite in general, one does not obtain any strong finiteness results.

\subsection{Setting $\epsilon = 1$}

Let us define $\ell : \WTits \rightarrow \ZZ$ by composing $\ell_\epsilon$ with the map $\ZZ \oplus \ZZ \epsilon \rightarrow \ZZ$ given by sending $\epsilon$ to $1$. We will prove the following, which was conjectured in \cite{M}:
\begin{Theorem}{\label{thm:compatibility-of-ell-with-order}}
  The map $\ell$ is strictly compatible with the Bruhat order on $\WTits$ and the usual order on $\ZZ$. That is, if $x,y \in \WTits$, $x \leq y$ and $x \neq y$, then $\ell(x) < \ell(y)$.
\end{Theorem}
Because $\ZZ$ does have finiteness of intervals, this implies that all chains between two fixed elements of $\WTits$ must be finite and gives an explicit bound. In fact, we will prove the following stronger statement.
\begin{Theorem}
  \label{thm-length-formula-for-sbeta-inversion-set}
 Let $\pi^\mu w \in \WTits$ and $\beta[n]\in\tDelta^+$. Suppose $\pi^\mu w s_{\beta[n]} > \pi^\mu w$. Then:
  \begin{align}
 \ell(\pi^\mu w s_{\beta[n]}) =  \ell(\pi^\mu w) + \# \{ \gamma[m] \in \Inv(s_{\beta[n]}) \suchthat \pi^\mu w( \gamma[m]) > 0 \textand \pi^\mu w( -s_{\beta[n]}(\gamma[m])) > 0 \}.    
  \end{align}
  In particular, the set $\{ \gamma[m] \in \Inv(s_{\beta[n]}) \suchthat \pi^\mu w( \gamma[m]) > 0 \textand \pi^\mu w( -s_{\beta[n]}(\gamma[m])) > 0 \}   $ is finite. 
\end{Theorem}
For brevity, let us define
\begin{align}
\Inv^{++}_{\pi^\mu w }(s_{\beta[n]}) = \{ \gamma[m] \in \Inv(s_{\beta[n]}) \suchthat \pi^\mu w( \gamma[m]) > 0 \textand \pi^\mu w( -s_{\beta[n]}(\gamma[m])) > 0 \}.
\end{align}
Then $\Inv^{++}_{\pi^\mu w}(s_{\beta[n]})$ contains at least one element, namely $\beta[n]$. Thus Theorem \ref{thm-length-formula-for-sbeta-inversion-set} implies Theorem~\ref{thm:compatibility-of-ell-with-order}.

\subsection{Some explicit formulas for length}

\begin{Proposition}
   \label{prop-height-formula}
For any $\mu\in\TitsCone$,
  \begin{align}
    \label{eq:15}
    \ell(\pi^\mu) = 2 \cdot  \height(\mu) - \sum_{\gamma\in\Delta^+: \langle \mu, \gamma \rangle < 0} \langle \mu, 2\gamma \rangle.
  \end{align}
\end{Proposition}
\begin{proof}
Because $\mu$ is in the Tits cone, there is some $w \in W$ such that $w(\mu)$ is dominant. Then, by definition, $\ell(\pi^\mu) = \langle w(\mu), 2\rho \rangle = \langle \mu, 2w^{-1}(\rho) \rangle $. Using \eqref{E:rho-wrho}, we obtain
\begin{align}
2\cdot\height(\mu) - \ell(\pi^\mu) = 2\langle \mu,\rho-w^{-1}(\rho)\rangle = \sum_{\eta\in\Inv(w)} \langle \mu, 2\eta\rangle
\end{align}

Suppose $\eta \in \Inv(w)$. Then $\langle \mu, \eta \rangle = \langle w(\mu), w(\eta) \rangle \leq 0$ because $w(\mu)$ is dominant and $w(\eta)$ is negative. Conversely, suppose $\gamma\in\Delta^+$ and $\langle \mu, \gamma \rangle < 0$. Then $\langle w(\mu), w(\gamma) \rangle < 0$, which implies that $w(\gamma)$ is negative. Therefore:
\begin{align}
  \label{eq:18}
  \sum_{\eta \in \Inv(w)} \langle \mu, 2 \eta \rangle = \sum_{\gamma\in\Delta^+ : \langle \mu, \gamma \rangle < 0} \langle \mu, 2\gamma \rangle
\end{align}
Note that the set over which we are summing on the right is a subset of that on the left, but the complement contributes zero to the sum.
\end{proof}

From the definition of $\ell$ and Proposition \ref{prop-height-formula}, we immediately obtain the following formula for the length of an arbitrary element $\pi^\mu w\in \WTits$:
\begin{align}
  \label{eq:19}
  \ell(\pi^\mu w) = 2 \cdot \height(\mu) + \sum_{\eta\in\Delta^+} 
  \begin{cases}
    - \langle \mu, 2 \eta \rangle & \textif \langle \mu, \eta \rangle < 0 \textand \eta \notin \Inv(w^{-1}) \\
    - \langle \mu, 2 \eta \rangle -1 & \textif \langle \mu, \eta \rangle < 0 \textand \eta \in \Inv(w^{-1}) \\
    1  & \textif \langle \mu, \eta \rangle \geq 0 \textand \eta \in \Inv(w^{-1}) \\
    0  & \textif \langle \mu, \eta \rangle \geq 0 \textand \eta \notin \Inv(w^{-1}). 
  \end{cases}
\end{align}

\section{Maps between double-affine inversion sets}

In order to prove Theorem~4.3, we need to introduce certain maps between inversion sets.

\subsection{The analagous problem for Weyl groups}
We believe that the discussion below is new even for Weyl groups, so let us first consider the analogous problem for the Weyl group $W$ of $\bG$. Let $w \in W$ and $\beta\in\Delta^+$, and let us assume $w(\beta) > 0$, i.e., $w s_\beta > w$ in the Bruhat order. Then $\ell(ws_\beta) > \ell(w)$. The problem that we wish to consider is:

\begin{enumerate}
\item Give an explicit injection $\Inv(w^{-1}) \hookrightarrow \Inv(s_\beta w^{-1})$.
\item Give an explicit bijection between the complement of the image above and the set $\{\gamma \in \Inv(s_\beta) \suchthat w(\gamma) > 0 \textand w(\iota(\gamma)) > 0 \}$, where $\iota = -s_\beta$.
\end{enumerate}

\subsubsection{The solution}

The following gives the solution to the first part of the problem.
\begin{Proposition}
  \label{prop-phi-is-injection}
Assume $w s_\beta > w$. Let $\eta \in \Inv(w^{-1})$. If $s_\beta w^{-1}(\eta) < 0$, then we define $\phi(\eta) = \eta$. If $s_\beta w^{-1}(\eta) > 0$, then we define $\phi(\eta) = s_{w(\beta)}(\eta)$. 
This rule defines an injection $\phi: \Inv(w^{-1}) \hookrightarrow \Inv(s_\beta w^{-1})$.  
\end{Proposition}

\begin{proof}
Suppose $s_\beta w^{-1}(\eta)>0$. We need to check that $\phi(\eta) \in \Inv(s_\beta w^{-1})$.
Let us write $\zeta = s_{w(\beta)}(\eta) = \phi(\eta)$. First, let us check that $\eta > 0$. We have assumed that $s_\beta w^{-1}(\eta) > 0$, equivalently $-s_\beta w^{-1}(\eta) < 0$. We also have $-w^{-1}(\eta) > 0$. So $-s_\beta w^{-1}(\eta) = s_\beta( - w^{-1}(\eta)) = - w^{-1}(\eta)  - \langle \betacheck, - w^{-1}(\eta) \rangle \beta $. Because $-w^{-1}(\eta) > 0$, this is a negative root only if 
$\langle \betacheck,  w^{-1}(\eta) \rangle < 0$. Now, let us compute $s_{w(\beta)}(\eta) = \eta - \langle w(\betacheck), \eta \rangle w(\beta)$. Because $\langle w(\betacheck), \eta \rangle = \langle \betacheck,  w^{-1}(\eta) \rangle < 0$ and $w(\beta)>0$ by hypothesis, we have $s_{w(\beta)}(\eta) > 0$.

Second, we compute $s_\beta w^{-1}(\zeta) = s_\beta w^{-1} w s_\beta w^{-1} (\eta) = w^{-1} (\eta) < 0$. So $\phi$ defines a map from $\Inv(w^{-1})$ to $\Inv(s_\beta w^{-1})$. 

Finally, we check that $\phi$ is an injection. Let $\eta, \tilde{\eta} \in \Inv(w^{-1})$. Suppose $\phi(\eta) = \phi(\tilde{\eta})$. If the sign of $s_\beta w^{-1}(\eta)$ and $s_\beta w^{-1}( \tilde{\eta})$ are the same, then it is clear that $\eta = \tilde{\eta}$ by the definition of $\phi$. So let assume $s_\beta w^{-1}(\eta) < 0$ and $s_\beta w^{-1}(\tilde{\eta}) > 0$. Then $\phi(\eta) = \eta$ and $\phi(\tilde{\eta}) = s_{w(\beta)}(\tilde{\eta})$. Our assumption is then that $\eta = s_{w(\beta)}(\tilde{\eta})$. But $w^{-1}(\eta) < 0$, while $w^{-1}s_{w(\beta)}(\tilde{\eta}) = s_{\beta}w^{-1}(\tilde{\eta}) > 0$. 
\end{proof}

We know that $\ell(ws_\beta) = \ell(w) + \# \{ \gamma \in \Inv(s_\beta) \suchthat w(\gamma) > 0 \textand w(\iota(\gamma)) > 0 \}$ (here $\iota = - s_\beta$). Therefore, it makes sense to ask for a natural bijection between the complement of the image of $\phi$ and $\{ \gamma \in \Inv(s_\beta) \suchthat w(\gamma) > 0 \textand w(\iota(\gamma)) > 0 \}$.

\begin{Proposition}
  \label{prop-psi-is-an-injection}
  Let $\gamma \in \Inv(s_\beta)$ such that $w(\gamma) > 0$ and $- w s_\beta (\gamma) > 0$. Then define $\psi(\gamma) = w(\gamma)$. This defines an injection $\psi :\{ \gamma \in \Inv(s_\beta) \suchthat w(\gamma) > 0 \textand w(-s_\beta(\gamma)) > 0 \} \hookrightarrow \Inv(s_\beta w^{-1})$.
\end{Proposition}

\begin{proof}
 Let $\gamma$ be as in the statement. By assumption, we have $w(\gamma) > 0$, and $s_\beta w^{-1} w (\gamma) = s_\beta(\gamma) < 0$. So $\psi(\gamma) \in \Inv(s_\beta w^{-1})$. Clearly $\psi$ is injective.
\end{proof}

\begin{Proposition}
  \label{prop-phi-psi-disjoint}
  The images of $\phi$ and $\psi$ are disjoint.
\end{Proposition}

\begin{proof}
Let $\gamma \in \Inv(s_\beta)$ such that $w(\gamma) > 0$ and $-  w( s_\beta(\gamma)) >0 $. Then $\psi(\gamma) = w(\gamma)$. Let $\eta \in \Inv(w^{-1})$. 
Let us assume $\psi(\gamma) = \phi(\eta)$.

For the first case, suppose $s_\beta w^{-1} < 0 $. Then $\phi(\eta) = \eta$. Our assumption is then that $\eta = w(\gamma)$, which implies $w^{-1}(\eta) = \gamma$. This is a contradiction since $w^{-1}(\eta) < 0$ while $\gamma > 0$.

For the second case, suppose $s_\beta w^{-1} > 0$. Then $\phi(\eta) = s_{w(\beta)}(\eta)$. Our assumption is then $s_{w(\beta)}(\eta) = w(\gamma)$, which implies $\gamma = s_{\beta} w^{-1} (\eta) > 0$. We also have the assumption that $- w (s_{\beta}(\gamma)) > 0$, which translates to $-\eta > 0$, which is again a contradiction.
\end{proof}

\begin{Corollary}
  \label{cor-phi-psi-cover-everything}
  \begin{align}
    \label{eq:17}
\Inv(s_\beta w^{-1})  = \im \phi \sqcup \im \psi
  \end{align}
\end{Corollary}
\begin{proof}
Both sides have the same cardinality by the length formula mentioned above; this gives the proof immediately. However, we would like to give a proof that avoids counting because the relevant sets need not be finite in the double-affine case.
  
Suppose $\theta \in \Inv(s_\beta w^{-1})$, then there are three cases.

\noindent Case 1: If $w^{-1}(\theta) < 0$, then let $\eta = \theta$, and we have $\phi(\eta) = \eta = \theta$.

\noindent Case 2: If $w^{-1}(\theta) > 0$ and $ s_{w(\beta)}(\theta) > 0$, then let $\eta = s_{w(\beta)}(\theta)$. Then $w^{-1}(\eta) = s_\beta w^{-1}(\theta) < 0$, and $s_\beta w^{-1} (\eta) = w^{-1}(\theta) > 0$. So we have $\phi(\eta) = s_{w(\beta} (\eta) = \theta$.

\noindent Case 3: If $w^{-1}(\theta) > 0$ and $ s_{w(\beta)}(\theta) < 0$, then let $\gamma = w^{-1}(\theta)$. Then $s_\beta (\gamma) = s_\beta w^{-1}(\theta) < 0$. So $\gamma \in \Inv(s_\beta)$. We have $w(\gamma) = \theta > 0$ and $- w(s_\beta(\gamma)) = - s_{w(\beta)}(\theta) > 0$. So $\psi(\gamma) = w(\gamma) = \theta$. 
\end{proof}

\subsection{The statements in the $\bG$-affine case}
\label{SS:double-affine-maps}
We can immediately generalize our solution to the $\bG$-affine case.
Let $\pi^\mu w \in \WTits$ and let $\beta[n]\in\tDelta^+$ be a positive $\bG$-affine root such that:
\begin{itemize}
\item $\pi^\mu w s_{\beta[n]} \in \WTits$
\item $\pi^\mu w(\beta[n]) > 0$
\end{itemize}
The constructions and proofs of the previous subsection carry over without change once we substitute $\pi^\mu w$ for $w$ and $\beta[n]$ for $\beta$. The explicit translations are as follows:

\begin{Proposition}
  For each $\eta[m] \in \Inv(w^{-1} \pi^{-\mu})$ define:
  \begin{align}
    \label{eq:53}
    \phi(\eta[m]) = 
    \begin{cases}
      \eta[m] & \textif s_{\beta[n]} w^{-1} \pi^{-\mu} ( \eta[m] ) < 0 \\
      \pi^\mu w s_{\beta[n]} w^{-1} \pi^{-\mu} ( \eta[m]) & \textif s_{\beta[n]} w^{-1} \pi^{-\mu} ( \eta[m] ) > 0.  
    \end{cases}
  \end{align}
This defines an injection $\phi : \Inv(w^{-1} \pi^{-\mu}) \hookrightarrow \Inv(s_{\beta[n]} w^{-1} \pi^{-\mu})$.
\end{Proposition}

\begin{Proposition}
  \label{prop-psi-is-injective-double-affine}
  For each $\gamma[m] \in \Inv^{++}_{\pi^\mu w}(s_{\beta[n]})$, define $\psi(\gamma[m]) = \pi^\mu w (\gamma[m])$. Then this defines an injection:
  \begin{align}
    \label{eq:54}
    \psi : \Inv^{++}_{\pi^\mu w} (s_{\beta[n]}) \hookrightarrow \Inv(s_{\beta[n]} w^{-1} \pi^{-\mu})
  \end{align}
\end{Proposition}

\begin{Proposition}
  The images of $\phi$ and $\psi$ are disjoint.
\end{Proposition}

\begin{Corollary}
  \begin{align}
    \label{eq:55}
    \Inv(s_{\beta[n]} w^{-1} \pi^{-\mu}) = \im(\phi) \sqcup \im(\psi)
  \end{align}
\end{Corollary}

\section{Some results on inversion sets and proof of Theorem \ref{thm-length-formula-for-sbeta-inversion-set}}

\subsection{Explicit computation of double-affine inversion sets}

%For $\pi^\mu w\in \WTits$ and $\eta[m]\in\tDelta^+$, we compute
%\begin{align*}
  %\label{eq:56}
%  w^{-1} \pi^{-\mu} ( \eta[m]) = \sgn(m) w^{-1} \pi^{-\mu} (\eta + m \pi) = 
%\sgn(m) \left( w^{-1}(\eta) + (m - \langle \mu, \eta \rangle ) \pi \right)
%\end{align*}
%If $\langle \mu, \eta \rangle \geq 0$, then $m \geq 0$ is a necessary condition for this double-affine root to be negative. Similarly, if $\langle \mu, \eta \rangle < 0$, then $m < 0$ is a necessary condition for $\eta[m]$ to be inverted by $\pi^\mu w$.

By direct computation, one finds that
\begin{align}
  \label{eq:46}
  \Inv(w^{-1}\pi^{-\mu}) = \left\{ \eta[m]\in\tDelta^+ \suchthat  
  \begin{cases}
     \langle \mu, \eta \rangle \leq m < 0 & \textif \langle \mu, \eta \rangle < 0 \textand \eta \notin \Inv(w^{-1}) \\
     \langle \mu, \eta \rangle < m < 0 & \textif \langle \mu, \eta \rangle < 0 \textand \eta \in \Inv(w^{-1}) \\
     0 \leq m \leq \langle \mu, \eta \rangle  & \textif \langle \mu, \eta \rangle \geq 0 \textand \eta \in \Inv(w^{-1}) \\
     0 \leq m < \langle \mu, \eta \rangle  & \textif \langle \mu, \eta \rangle \geq 0 \textand \eta \notin \Inv(w^{-1}) 
  \end{cases}
                                                \right\}. 
\end{align}

Let $\cS \subset \Delta^+$ be a finite subset of the positive single-affine real roots, and let us define 
\begin{align*}
\Inv_{\cS}(w^{-1} \pi^{-\mu}) = \{ \eta[m] \suchthat \eta \in \cS \textand \eta[m] \in \Inv(w^{-1} \pi^{-\mu}) \}.
\end{align*}
Then $\Inv_{\cS}(w^{-1} \pi^{-\mu})$ is finite and we have:
\begin{align}
  \label{E:invS}
  \# \Inv_{\cS}(w^{-1} \pi^{-\mu}) = \sum_{\eta \in \cS} \left( |\langle \mu, \eta \rangle| + 
  \begin{cases}
    -1 & \textif \langle \mu, \eta \rangle < 0  \textand \eta \in \Inv(w^{-1}) \\
    +1 & \textif \langle \mu, \eta \rangle \geq 0  \textand \eta \in \Inv(w^{-1}) \\
    0 & \text{ otherwise}
  \end{cases}
        \right).
\end{align}

\subsection{Finiteness of $\Inv^{++}_{\pi^\mu w} (s_{\beta[n]})$}

\begin{Theorem}
Let us assume that $\pi^\mu w$ and $\beta[n]$ are as in Section~\ref{SS:double-affine-maps}, i.e., $\pi^\mu w(\beta[n]) > 0$ and $ \pi^\mu w s_{\beta[n]} \in \WTits$. Then the set $\Inv^{++}_{\pi^\mu w} (s_{\beta[n]})$ is finite. 
\end{Theorem}

\begin{proof}
  
By Proposition \ref{prop-psi-is-injective-double-affine}, we can identify this set with its image under the map $\psi$. By the proof of Corollary \ref{cor-phi-psi-cover-everything}, we can identify the image of $\psi$ with the set of positive $\bG$-affine real roots $\theta[m]$ such that:
\begin{itemize}
\item $s_{\beta[n]} w^{-1} \pi^{-\mu} \left(\theta[m] \right) < 0$
\item $w^{-1} \pi^{-\mu} \left(\theta[m] \right) > 0$ 
\item $ \pi^{\mu} w s_{\beta[n]} w^{-1} \pi^{-\mu} \left(\theta[m] \right) < 0 $
\end{itemize}
By \eqref{eq:46}, to show that $\Inv^{++}_{\pi^\mu w} (s_{\beta[n]})$ is finite it suffices to show that only finitely many $\theta$ can occur.

For the first condition, we compute:
\begin{align*}
  %\label{eq:58}
s_{\beta[n]} w^{-1} \pi^{-\mu} \left(\theta[m] \right)  &= 
\pi^{n \betacheck} s_\beta w^{-1} \pi^{-\mu} \left(\theta[m] \right) \\
 &= \pi^{n \betacheck}  \sgn(m) \left(  s_\beta w^{-1}(\theta) + ( m - \langle \mu, \theta \rangle ) \pi \right)\\
 &= 
\sgn(m) \left(  s_\beta w^{-1}(\theta) + ( m - \langle \mu, \theta \rangle  - n \langle w(\betacheck),\theta \rangle) \pi \right)\\
 &=
\sgn(m) \left(  s_\beta w^{-1}(\theta) + ( m - \langle \mu + n w(\betacheck), \theta \rangle  ) \pi \right). 
\end{align*}
For the second condition:
\begin{align*}
  %\label{eq:59}
w^{-1} \pi^{-\mu} \left(\theta[m] \right)  = 
\sgn(m) (w^{-1}(\theta) + (m - \langle \mu, \theta \rangle) \pi ).
\end{align*}
For the third condition:
\begin{align*}
  %\label{eq:60}
 \pi^{\mu} w s_{\beta[n]} w^{-1} \pi^{-\mu} \left(\theta[m] \right)  &=   
   \pi^{\mu} w  
\sgn(m) \left(  s_\beta w^{-1}(\theta) + ( m - \langle \mu + n w(\beta), \theta \rangle  ) \pi \right)\\
&= \sgn(m) \pi^{\mu} (s_{w(\beta)}(\theta) + ( m - \langle \mu + n w(\betacheck), \theta \rangle  ) \pi )\\
&= \sgn(m) (s_{w(\beta)}(\theta) + ( m - \langle \mu + n w(\betacheck), \theta \rangle  + \langle \mu, s_{w(\beta)}(\theta)\rangle) \pi ).
\end{align*}

Because $\mu + n w(\betacheck) \in \TitsCone$ by assumption, $\langle \mu + n w(\betacheck), \theta \rangle \geq 0$ for almost all $\theta$. As we are interested in proving finiteness of the set of $\theta$ that occur, we can go ahead and assume  $\langle \mu + n w(\betacheck), \theta \rangle \geq 0$. 
Then the first condition necessitates that $m \geq 0$. The second condition requires that $\langle \mu, \theta \rangle \leq m$.

To handle the third condition, we compute:
\begin{align*}
  %\label{eq:63}
 \langle \mu + n w(\betacheck), \theta \rangle  - \langle \mu, s_{w(\beta)}(\theta)\rangle &=
\langle \mu, \theta \rangle + n \langle  w(\betacheck), \theta \rangle - \langle \mu, \theta \rangle + \langle \mu, w(\beta) \rangle \langle w(\betacheck) , \theta \rangle\\
&= (n + \langle \mu, w(\beta) \rangle ) \langle w(\betacheck), \theta \rangle.
\end{align*}
We see that the third condition necessitates that $m \leq 
(n + \langle \mu, w(\beta) \rangle ) \langle w(\betacheck), \theta \rangle$.
Note that $n + \langle \mu, w(\beta) \rangle $ does not depend on $\theta$. 

\comment{
Let us further assume the following hypotheses.
\begin{Hypothesis}
The pairing between roots and coroots is uniformly bounded above, i.e., assume that $G$ has the property that there exists a positive integer $N$ such that
\begin{align}
  \langle \gamma^\vee, \beta \rangle < N 
\end{align}
whenever $\gamma^\vee$ is a coroot and $\beta$ is a root.
\end{Hypothesis}
\begin{Hypothesis}
Almost all roots pair with $\mu$ to large values, i.e., for any $N > 0$, the set of roots $\theta$ such that $\langle \mu, \theta \rangle < N$ is finite.
\end{Hypothesis}

The first hypthesis is true in the affine case. Perhaps it is true more generally, but I don't immediately know. The second hypothesis is true whenever $\mu$ has strictly positive level (level zero isn't interesting). For the general Kac-Moody case, our results carry over once we've verified these hypotheses (or have otherwise verified finiteness of $\Inv^{++}_{\pi^\mu w}(s_{\beta[n})$).
}

The second and third conditions imply that $\langle \mu, \theta \rangle \leq m \leq (n + \langle \mu, w(\beta) \rangle ) \langle w(\betacheck), \theta \rangle$. Since
\begin{align*}
\mu - (n + \langle \mu, w(\beta) \rangle) w(\betacheck) = s_{w(\beta)}(\mu + n w(\betacheck)),
\end{align*}
there exist $m$ in this range if and only if $\langle s_{w(\beta)}(\mu + n w(\betacheck)), \theta \rangle \le 0$. By assumption, $\mu + n w(\betacheck) \in \TitsCone$ and hence $\nu = s_{w(\beta)}(\mu + n w(\betacheck)) \in \TitsCone$. The set of $\theta$ such that $\langle \nu, \theta \rangle < 0$ therefore must be finite. While $\langle \nu, \theta \rangle = 0$ is possible for infinitely many $\theta$, we can have $\theta[m] \in \Inv^{++}_{\pi^\mu w}(s_{\beta[n]})$ in this case only if $s_{w(\beta)}(\theta)<0$. Since $\Inv(s_{w(\beta)})$ is finite, we have our result.
\end{proof}

% By the hypotheses, this can only hold for a finite set of $\theta$. So we have our finiteness result.

\subsection{Putting it all together}

\begin{proof}[Proof of Theorem \ref{thm-length-formula-for-sbeta-inversion-set}]
Let $\cS \subset \Delta^+$ be a finite subset of the positive single-affine roots such that:
\begin{enumerate}
\item $\cS$ is invariant under $|s_{w(\beta)}|$
\item $\cS$ contains $\Inv(s_{w(\beta)})$ and $\Inv(w^{-1})$
\item $\cS$ contains all $\eta$ such that there exists $m$ such that $\eta[m] \in \im(\psi)$.
\item $\cS$ contains $\eta$ such that $\langle \mu, \eta \rangle < 0$.
\item $\cS$ contains $\eta$ such that $\langle \mu + n w(\betacheck), \eta \rangle < 0$.
\end{enumerate}
The fact that such an $\cS$ exists follows from the finiteness of $\Inv^{++}_{\pi^\mu w}(s_{\beta[n]})$ and the assumption that $\mu,\mu+nw(\betacheck)\in\TitsCone$. Let us define $\cS^c = \Delta^+ \backslash \cS$. Then $\cS^c$ is also invariant under $|s_{w(\beta)}|$.

Let us also observe that if $\eta[m] \in \Inv(w^{-1}\pi^{-\mu})$, then $\phi(\eta[m]) = \eta[m]$ or $\phi(\eta[m]) = |s_{w(\beta)}|(\eta)[p]$ for some integer $p$. This implies, that $\phi$ restricts to a map:
\begin{align*}
  %\label{eq:64}
  \phi: \Inv_{\cS^c} (w^{-1} \pi^{-\mu}) \rightarrow \Inv_{\cS^c}(s_{\beta[n]} w^{-1} \pi^{-\mu}).
\end{align*}
By the third condition that $\cS$ must satisfy, we see that this map must be a bijection. We deduce that
\begin{align*}
  %\label{eq:61}
  \# \Inv^{++}_{\pi^\mu w}(s_{\beta[n]}) = \# \im(\psi) = \# \Inv_{\cS}(s_{\beta[n]} w^{-1} \pi^{-\mu}) - \# \Inv_{\cS}(w^{-1} \pi^{-\mu}).
\end{align*}

It remains to compute the right-hand side of the previous formula and show that it is equal to the difference $\ell(\pi^\mu w)-\ell(\pi^\mu ws_{\beta[n]})$. By \eqref{E:invS}, we have:
\begin{align*}
  %\label{eq:057}
  \# \Inv_{\cS}(w^{-1} \pi^{-\mu}) = \sum_{\eta \in \cS} \left( |\langle \mu, \eta \rangle| + 
  \begin{cases}
    -1 & \textif \langle \mu, \eta \rangle < 0  \textand \eta \in \Inv(w^{-1}) \\
    +1 & \textif \langle \mu, \eta \rangle \geq 0  \textand \eta \in \Inv(w^{-1}) \\
    0 & \text{ otherwise}
  \end{cases}
        \right).
\end{align*}
Let us write $s_{\beta[n]} w^{-1} \pi^{-\mu} = \pi^{n\betacheck} s_\beta w^{-1} \pi^{-\mu} = s_\beta w^{-1} \pi^{ - (\mu + n w(\betacheck) )}$. Also by \eqref{E:invS}, we have: 
\begin{align*}
  %\label{eq:62}
&\# \Inv_{\cS}(s_{\beta[n]} w^{-1} \pi^{-\mu})= \# \Inv_{\cS}( s_\beta w^{-1} \pi^{ - (\mu + n w(\betacheck) )})\\ &=  
\sum_{\eta \in \cS} \left( |\langle \mu + n w(\betacheck), \eta \rangle| + 
  \begin{cases}
    -1 & \textif \langle \mu + n w(\betacheck), \eta \rangle < 0  \textand \eta \in \Inv(s_\beta w^{-1}) \\
    +1 & \textif \langle \mu + n w(\betacheck), \eta \rangle \geq 0  \textand \eta \in \Inv(s_\beta w^{-1}) \\
    0 & \text{ otherwise}
  \end{cases}
        \right).
\end{align*}

Using the length formulas, we have:
\begin{align*}
  %\label{eq:019}
  \ell(\pi^\mu w) &= 2 \cdot \height(\mu) + \sum_{\eta \in\Delta^+} 
  \begin{cases}
    - \langle \mu, 2 \eta \rangle & \textif \langle \mu, \eta \rangle < 0 \textand \eta \notin \Inv(w^{-1}) \\
    - \langle \mu, 2 \eta \rangle -1 & \textif \langle \mu, \eta \rangle < 0 \textand \eta \in \Inv(w^{-1}) \\
    1  & \textif \langle \mu, \eta \rangle \geq 0 \textand \eta \in \Inv(w^{-1}) \\
    0  & \textif \langle \mu, \eta \rangle \geq 0 \textand \eta \notin \Inv(w^{-1}) 
  \end{cases} \\
   &= 2 \cdot \height(\mu) - \sum_{\eta \in \cS} \langle \mu, \eta \rangle +  
\sum_{\eta \in \cS} \left( |\langle \mu, \eta \rangle| + 
  \begin{cases}
    -1 & \textif \langle \mu, \eta \rangle < 0  \textand \eta \in \Inv(w^{-1}) \\
    +1 & \textif \langle \mu, \eta \rangle \geq 0  \textand \eta \in \Inv(w^{-1}) \\
    0 & \text{ otherwise}
  \end{cases}
        \right) \\
&= 2 \cdot \height(\mu) - \sum_{\eta \in \cS} \langle \mu, \eta \rangle +   \# \Inv_{\cS}(w^{-1} \pi^{-\mu}).
\end{align*}
The second equality follows by the fourth and second conditions on $\cS$.
Similarly, using the fifth and second conditions, we can compute:
\begin{align}
  \label{eq:65}
  \ell(\pi^{\mu + n w(\betacheck)} w s_\beta) = 2 \cdot \height(\mu + n w(\betacheck)) - \sum_{\eta \in \cS}\langle \mu + n w(\betacheck), \eta \rangle + \# \Inv_{\cS}(s_{\beta[n]} w^{-1} \pi^{-\mu}).
\end{align}

%Now, let us compute:
%\begin{align}
%  \label{eq:66}
%  2 \cdot \height(\mu + n w(\beta)) - \sum_{\eta \in \cS}\langle \mu + n w(\beta), \eta \rangle = \\
%2 \cdot \height(\mu) - \sum_{\eta \in \cS} \langle \mu, \eta \rangle + n \cdot (2 \cdot \height(w(\beta)) - \sum_{\eta \in \cS} \langle w(\beta), \eta \rangle)
%\end{align}
By Propostion \ref{prop-the-height-formula} and the first two conditions on $\cS$, we have $2 \cdot \height(w(\betacheck)) = \sum_{\eta \in \cS} \langle w(\betacheck), \eta \rangle$.
%Therefore, 
%\begin{align}
%  \label{eq:67}
%  2 \cdot \height(\mu + n w(\beta)) - \sum_{\eta \in \cS}\langle \mu + n w(\beta), \eta \rangle = \\
%2 \cdot \height(\mu) - \sum_{\eta \in \cS} \langle \mu, \eta \rangle 
%\end{align}
%which implies:
Therefore:
\begin{align*}
  %\label{eq:68}
  \ell(\pi^{\mu + n w(\betacheck)} w s_\beta)  -  \ell(\pi^\mu w) = \# \Inv_{\cS}(s_{\beta[n]}w^{-1} \pi^{-\mu}) - \# \Inv_{\cS}(w^{-1} \pi^{-\mu}) =  \# \Inv^{++}_{\pi^\mu w} (s_{\beta[n]}).
\end{align*}
This completes the proof.
\end{proof}

\section{Classification of covers}

\subsection{General conjecture}
For $x,y \in \WTits$, let us write $x\lhd y$ to indicate a covering relation, i.e., $x\lhd y$ if and only $x<y$ and $\{z \suchthat x<z<y\}=\emptyset$.

By the definition of the order, a necessary condition for $ x \lhd y$ is that $y = x s_{\beta[n]}$ for some $\bG$-affine real root $\beta[n]$.
Motivated by the well-known characterization of covering relations for the Bruhat order on a Coxeter group, we make the following:

\begin{Conjecture}\label{conj:covers}
We have $x\lhd y$ if and only $ x < y$ and $\ell(y)=\ell(x)+1$.
\end{Conjecture}

Below we will give a positive answer to this question when $\bG$ is of untwisted affine ADE type. We will proceed by explicitly computing double-affine roots. 
%By Theorem~\ref{thm-length-formula-for-sbeta-inversion-set}, we know that if $x\lhd y$ then $\ell(y)-\ell(x)$ must be an odd positive integer, since the set on the right-hand side carries an involution (namely, $\iota=-s_{\beta[n]}$) with exactly one fixed point. 

\subsection{Explicit description of double-affine roots when $\bG$ is untwisted affine}

Let $\bG_0$ be a finite-type Kac-Moody group, and let $\bG$ be the untwisted affinization of $\bG_0$. Let us write $\delta$ for the minimal imaginary root for $\bG$. Now we refer to $\bG$-affine roots as double-affine roots.

Let $\beta$ be a positive root for $\bG_0$. Then for every pair $(r,n) \in \ZZ^2$, we define
\begin{align}
 \beta[r,n] = \sigma(r,n)(\beta + r\delta + n \pi) 
\end{align}
where the sign $\sigma(r,n) \in \{\pm1\}$ is defined to make the above expression a positive double-affine root. Explicitly, we define
\begin{align}
 \sigma(r,n) =
  \begin{cases}
    +1 & \textif n > 0 \textor n=0 \textand r\geq0 \\
    -1 & \textif n < 0 \textor n=0 \textand r<0. 
  \end{cases}
\end{align}

One immediate benefit of this indexing is the following simple formula:
\begin{align}
 s_{\beta[r,n]} = \pi^{nr\delta} \pi^{n\betacheck} t^{r\betacheck} s_\beta.
\end{align}

\begin{Caution}
 It is almost true that:
 \begin{align}
   \label{eq:38}
   \sigma(-x,-y) = -\sigma(x,y).
 \end{align}
 But this is not true when $(x,y)=(0,0)$.
\end{Caution}

\subsubsection{The action of reflections}

Let us compute
\begin{align}
  \label{eq:reflection-formula}
   s_{\beta[r,n]}(\beta[s,m]) 
   &= \beta[s,m] - 2 \sigma(s,m) \sigma(r,n) \beta[r,n]\\
   &= \sigma(s,m) ( \beta + s \delta + m \pi - 2( \beta + r \delta + n \pi) )\notag\\
   &= -\sigma(s,m) ( \beta + (2r-s)\delta + (2n-m)\pi)\notag\\
   &= -\sigma(s,m) \sigma(2r-s,2n-m) \beta[2r-s,2n-m]\notag.
\end{align}
Therefore, we have
\begin{align}
  |s_{\beta[r,n]}| (\beta[s,m]) = \beta[2r-s,2n-m].
\end{align}
That is, we can say that the action of $|s_{\beta[r,n]}|$ on pairs of integers indexing double-affine real roots $\beta[s,m]$ is exactly $180^\circ$ rotation about the point $(r,n)$. In particular this is true for the map $\iota$, which is the restriction of $|s_{\beta[r,n]}|$ to $\Inv(s_{\beta[r,n]})$.

\subsubsection{The action of arbitrary $x$ on double-affine roots}

Let $x \in \WTits$. Then we can write
\begin{align}
\label{eq:2}
x = \pi^{\ell \Lambda_0} \pi^{\mu} t^{\nu} w
\end{align}
where $\ell > 0$, $\mu$ and $\nu$ are finite coweights for $\bG_0$, and $w \in W_{\bG_0}$ (the finite Weyl group associated to $\bG_0$).

Then: 
\begin{align}
  \label{eq:1}
x(\beta[s,m]) &= \pi^{\ell \Lambda_0} \pi^{\mu} t^{\nu} w ( \beta[s,m])\\
              &=\sigma(s,m) (w(\beta) + (s + \langle \nu, w(\beta) \rangle) \delta + (m + \langle \mu, w(\beta) \rangle) + \ell  (s + \langle \nu, w(\beta) \rangle) ) \pi).\notag
\end{align}
Let us write $a = -\langle \nu, w(\beta) \rangle$ and $b = -\langle \mu, w(\beta) \rangle$. Then \eqref{eq:1} is equal to:
\begin{align}
  \label{eq:3}
&\sigma(s,m) (w(\beta) + (s - a) \delta + (m - b) + \ell  (s - a) ) \pi)\\
&=  \begin{cases}
\sigma(s,m) \cdot \sigma(s-a,m-b+\ell(s-a)) \cdot  w(\beta)[s-a, m-b + \ell(s-a)] & \textif{w(\beta) > 0} \\
\sigma(s,m) \cdot \sigma(a-s,b-m+\ell(a-s)) \cdot  -w(\beta)[a-s, b-m + \ell(a-s)] & \textif{w(\beta) < 0}. 
  \end{cases}\notag
\end{align}

\subsection{Untwisted affine ADE}
We will prove the following:

\begin{Theorem}
  \label{thm:ade-covering}
  Let $\bG_0$ be a simply-laced finite-type Kac-Moody group (i.e., ADE type). Let $\bG$ be its untwisted affinization. Then Conjecture \ref{conj:covers} is true for $\bG$.
 \end{Theorem}

 Let us write $(\cdot,\cdot)$ for the Weyl-invariant Euclidean inner product on the root space of $\bG_0$ such that all roots have square length of $2$. Then for any pair $\beta$ and $\theta$ of positive roots for $\bG_0$, we have:
 \begin{align}
   \label{eq:16}
 \langle \beta^\vee, \theta \rangle = ( \beta, \theta  ) 
 \end{align}
 the right hand side of this equation. Below, we will abuse notation and simple write $\langle \beta, \theta \rangle$ for the pairing $( \beta, \theta )$.

The following well-known fact is crucial for our argument.
\begin{Lemma}
  \label{lem:pairings-in-finite-ade}
  Let $\theta$ and $\beta$ be distinct positive roots for $\bG_0$ which is finite-type ADE. Then:
  \begin{align}
    \label{eq:9}
     \langle \beta, \theta \rangle  \in \{-1, 0, 1 \}. 
  \end{align}
\end{Lemma}
%\begin{proof}
% Because $\theta$ and $\beta$ have the same square length of $2$, $\langle \beta, \theta \rangle$ is equal to $2$ times the cosine of the angle between $\theta$ and $\beta$. Because $\theta$ and $\beta$ are distinct and both positive, we must have:
% \begin{align}
%   \label{eq:11}
%   -2 < \langle \beta, \theta \rangle < 2
% \end{align}
% The lemma follows then from the fact that $\langle \beta, \theta \rangle \in \ZZ$.
%\end{proof}

In order to prove Theorem~\ref{thm:ade-covering}, we need to show that for $x \in \WTits$ and $\beta[r,n]$ a double-affine real root, $x \lhd x s_{\beta[r,n]}$ implies $\# \Invplusplusx(s_{\beta[r,n]}) = 1$. Let us suppose that $x \rightarrow x s_{\beta[r,n]}$, and $\ell(x s_{\beta[r,n]}) - \ell(x) > 1$; we will show that $x s_{\beta[r,n]}$ is not a cover of $x$. This will be accomplished in Propositions~\ref{prop:not-cover} and \ref{prop:sl-two-cover-classification}.

Given double-affine roots $\beta[r,n]$ and $\theta[s,m]$ in ADE type, let us define:
\begin{align}
  \label{eq:31}
  \langle \beta[r,n], \theta[s,m] \rangle = \sigma(r,n)\sigma(s,m)  \langle \beta, \theta \rangle 
\end{align}
We then have the following lemma.
\begin{Lemma}
  \label{lem:reflection-formula-with-pairing}
  For double-affine roots $\beta[r,n]$ and $\theta[s,m]$ in ADE type we have.
  \begin{align}
    \label{eq:32}
  s_{\beta[r,n]}(\theta[s,m]) = \theta[s,m] - \langle \theta[s,m], \beta[r,n] \rangle \beta[r,n]
  \end{align}
\end{Lemma}

\begin{Proposition}\label{prop:not-cover}
  Suppose $\bG_0$ is simply-laced, $x \rightarrow x s_{\beta[r,n]}$, and 
$\ell(x s_{\beta[r,n]}) - \ell(x) > 1$. Suppose further, there exist some finite root $\theta$ such that $\theta \neq \beta$ and a pair $(s,m) \in \ZZ^2$ such that:
\begin{align}
  \label{eq:8}
  \theta[s,m] \in \Invplusplusx(s_{\beta[r,n]})
\end{align}
Then $x s_{\beta[r,n]}$ is not a cover of $x$.
\end{Proposition}
 
\begin{proof}
%Let $ \gamma[t,p] = - s_{\beta[r,n]} (\theta[s,m])$.
We claim that
\begin{align}
  \label{eq:12}
 x \rightarrow x s_{\theta[s,m]} \rightarrow x s_{\beta[r,n]} s_{\theta[s,m]} \rightarrow  x s_{\beta[r,n]} 
\end{align}
is a chain in the Bruhat order.

The relations $x \rightarrow x s_{\theta[s,m]}$ and $x s_{\beta[r,n]} s_{\theta[s,m]} \rightarrow  x s_{\beta[r,n]}$ follow from $\theta[s,m] \in \Invplusplusx(s_{\beta[r,n]})$. So we just need to prove $x s_{\theta[s,m]} \rightarrow x s_{\beta[r,n]} s_{\theta[s,m]} $.

First, we claim that $s_{\theta[s,m]}(\beta[r,n]) > 0$. We know $s_{\beta[r,n]}(\theta[s,m]) < 0$ and
\begin{align}
  \label{eq:13}
  s_{\beta[r,n]}(\theta[s,m]) = \theta[s,m] - \langle \theta[s,m], \beta[r,n] \rangle \beta[r,n]
\end{align}
by Lemma \ref{lem:reflection-formula-with-pairing}.
This implies:
\begin{align}
  \label{eq:14}
  \langle \theta[s,m], \beta[r,n] \rangle > 0
\end{align}
By Lemma \ref{lem:pairings-in-finite-ade}, we must have:
\begin{align}
  \label{eq:10}
  \langle \theta[s,m], \beta[r,n] \rangle = 1
\end{align}
So we compute:
\begin{align}
  \label{eq:20}
  s_{\theta[s,m]}(\beta[r,n]) = \beta[r,n] - \langle \beta[r,n], \theta[s,m] \rangle \theta[s,m] = \beta[r,n] - \theta[s,m] = 
- s_{\beta[r,n]}(\theta[s,m]) > 0
\end{align}
Here we use the fact that $ \langle \theta[s,m], \beta[r,n] \rangle = \langle \beta[r,n], \theta[s,m] \rangle$.

Let us compute $x s_{\beta[r,n]} s_{\theta[s,m]} = x s_{\theta[s,m]} s_{\theta[s,m]} s_{\beta[r,n]} s_{\theta[s,m]} = x s_{\theta[s,m]} s_{s_{\theta[s,m]}(\beta[r,n])}$. Because we have shown that $s_{\theta[s,m]}(\beta[r,n]) > 0$, $x s_{\theta[s,m]} \rightarrow x s_{\beta[r,n]} s_{\theta[s,m]} $ if and only if $x s_{\theta[s,m]} s_{\theta[s,m]}(\beta[r,n])) = x(\beta[r,n])$ is positive; this is exactly our initial hypothesis. 
\end{proof}

Therefore, we are now reduced to the ``rank-one'' case when $x \rightarrow x s_{\beta[r,n]}$, $\ell(x s_{\beta[r,n]}) - \ell(x) > 1$, and all elements of $\Invplusplusx(s_{\beta[r,n]})$ are of the form $\beta[s,m]$ for some $(s,m) \in \ZZ^2$. We will handle this in the next section.

% \begin{proof}[Proof of Conjecture~\ref{C:covers} for simply-laced $G$]
% Suppose the set
% \begin{align*}
% \Inv^{++}_x(s_{\beta[n]})=\{\theta[m]\in\Inv(s_{\beta[n]}) : \text{$x(\theta[m])>0$ and $x(-s_{\beta[n]}(\theta[m]))>0$}\}
% \end{align*}
% contains more than one element. Let $\theta[m]$ be an element of this set different from $\beta[m]$. Then we have $x<x':=xs_{\theta[m]}$ since $x(\theta[m])>0$ and $y':=ys_{\iota(\theta[m])}<y$ since $y(\iota(\theta[m]))=-x(\theta[m])<0$. Notice that
% \begin{align*}
% y'=xs_{\beta[n]}s_{\iota(\theta[m])}=xs_{-\theta[m]}s_{\beta[n]}=x's_{\beta[n]}.
% \end{align*}
% We have $x'(\beta[n])=xs_{\theta[m]}(\beta[n])$ and
% \begin{align*}
% s_{\theta[m]}(\beta[n])=\beta[n]-\theta[m]=-s_{\beta[n]}(\theta[m])=\iota(\theta[m])
% \end{align*}
% since $G$ is assumed simply-laced. Therefore $x'(\beta[n])=x(\iota(\theta[m]))>0$ and we have $x<x'<y'<y$. This proves that, unless $\ell(y)-\ell(x)=1$ can always deduce $x<y$ from ``shorter'' relations. This completes the proof.
% \end{proof}

\section{The rank-one case}

Let us consider $x$, $\beta$, and a pair $(r,n) \in \ZZ^2$ as above, and let us suppose that
\begin{align}
  \label{eq:104}
  x(\beta[r,n]) > 0 \\
  \ell(xs_{\beta[r,n]}) - \ell(x) > 1
\end{align}
and all double-affine roots in $\Invplusplusx(s_{\beta[r,n]})$ are of the form $\beta[s,m]$ for some $(s,m) \in \ZZ^2$.
The second condition is equivalent by Theorem \ref{thm-length-formula-for-sbeta-inversion-set} to:
\begin{align}
  \label{eq:105}
  \#\Invplusplusx(s_{\beta[r,n]}) > 1.
\end{align}
To complete the proof of Theorem \ref{thm:ade-covering}, we must prove:
\begin{Proposition}
\label{prop:sl-two-cover-classification}
The element $xs_{\beta[r,n]}$ is \emph{not} a cover of $x$.
\end{Proposition}

Let us first make the following simplifying assumptions:
\begin{itemize}
\item $w(\beta) > 0$
\item $\sigma(r,n) > 0$
%\item $\sigma(a,b) > 0$ \dinakar{We want to drop this assumption.}
\end{itemize}
The other situations are handled by arguments similar to what we present below.
We will divide the proof into the following cases.
\begin{itemize}
\item {\bf Case 1:} $\sigma(r,n-1) = -1$
\item {\bf Case 2:} $n > 0$ and $ \beta[r,n-1] \in \Invplusplusx(s_{\beta[r,n]})$
\item {\bf Case 3:} $n > 0$ and $ \beta[r,n-1] \notin \Invplusplusx(s_{\beta[r,n]})$
\end{itemize}

\begin{Remark}
 Supposing that $\sigma(r,n) =1$ and plotting the pairs $(r,n)$ such that $x(\beta[r,n]) > 0$, we get a polyhedral region corresponding to the condition:
 \begin{align}
   \label{eq:48}
 \sigma(r-a,n-b+\ell(r-a)) > 0
 \end{align}
 Case 2 above corresponds to the ``interior'' of that region, while Cases 1 and 3 correspond to the ``boundary'' of that region. 
\end{Remark}

We will use the following, which is evident from \eqref{eq:3}:
\begin{Lemma}
  \label{lem:pos-prop-of-inv-plusplus}
  For all $i,j\geq 0$:
  \begin{align}
  \label{eq:55}
  x(\beta[r+i,n+j]) > 0.
\end{align}
\end{Lemma}

\subsection{{\bf Case 1:} $\sigma(r,n-1) = -1$}

There are three subcases.

\subsubsection{Subcase: $r = 0$}
This subcase does not occur because $r=0$ implies that $n=0$ and that $\# \Invplusplusx(s_{\beta[r,n]}) = 1$.

\subsubsection{Subcase: $r > 0$}
    
In this subcase $n=0$. We must have
\begin{align}
  \label{eq:59}
 \beta[r-1,n] \in \Invplusplusx(s_{\beta[r,n]}) 
\end{align}
because otherwise $\# \Invplusplusx(s_{\beta[r,n]}) = 1$.
In this case, 
\begin{align}
  \label{eq:58}
   x \rightarrow x s_{\beta[r-1,n]} \rightarrow x s_{\beta[r-1,n]}s_{\beta[-1,n]} \rightarrow x s_{\beta[r,n]}
\end{align}
is a three-term chain.

\begin{proof}
 
For the first term of the chain, we have $x( \beta[r-1,n]) > 0$ by assumption.

For the second term, we compute
\begin{align}
  \label{eq:61}
x s_{\beta[r-1,n]} (\beta[-1,n]) = x(\beta[2r-1,n]) >0
\end{align}
because of Lemma \ref{lem:pos-prop-of-inv-plusplus} and the fact that $r \geq 1$.

Finally, for the third term, we compute:
\begin{align}
  \label{eq:63}
  xs_{\beta[r-1,n]}s_{\beta[-1,n]}(\beta[0,n])
  &= xs_{\beta[r-1,n]}s_{\beta[-1,n]}s_{\beta[0,n]}( - \beta[0,n])\\
  &= xs_{\beta[r,n]}(- \beta[0,n]) = x(\beta[2r,n])>0.\notag
\end{align}

\end{proof}

\subsubsection{Subcase: $r < 0$}

Note that this subcase implies that $n=1$. Let us state some lemmas.

\begin{Lemma}
  In this subcase, if 
  \begin{align}
    \label{eq:64}
      \# \Invplusplusx(s_{\beta[r,1]}) > 1
  \end{align}
  then either
  \begin{align}
    \label{eq:65}
   \beta[r-1,1] \in \Invplusplusx(s_{\beta[r,1]})
  \end{align}
  or:
  \begin{align}
    \label{eq:66}
   \beta[2r,2] \in \Invplusplusx(s_{\beta[r,1]}).
  \end{align}
\end{Lemma}

\begin{Lemma}
Let us make the assumptions of this subcase. Suppose
  \begin{align}
    \label{eq:064}
      \# \Invplusplusx(s_{\beta[r,1]}) > 1
  \end{align}
  and:
  \begin{align}
    \label{eq:065}
   \beta[r-1,1] \notin \Invplusplusx(s_{\beta[r,1]})
  \end{align}
  Then $r = -1$, and by the previous lemma:
  \begin{align}
    \label{eq:67}
    \beta[-2,2]  \in \Invplusplusx(s_{\beta[r,1]}).
  \end{align}
\end{Lemma}

\noindent {\it Subsubcase:} $\beta[r-1,1] \in \Invplusplusx(s_{\beta[r,1]})$.

Let us choose $c>0$ to be the largest integer such that:
\begin{align}
  \label{eq:60}
x(\beta[r-c,1]) > 0
\end{align}
Then we claim that 
\begin{align}
  \label{eq:62}
  x \rightarrow x s_{\beta[r-c,1]} \rightarrow x s_{\beta[r-c,1]} s_{\beta[r,1]} \rightarrow x s_{\beta[r,1]}
\end{align}
is a three-term chain.
\begin{proof}
  We have $x(\beta[r-c,1]) > 0$ by construction.

For the second term, we have
\begin{align}
  \label{eq:68}
x s_{\beta[r-c,1]} (\beta[r,1]) = x ( - \beta[r-2c,1]) >0
\end{align}
because $2c > c$.

For the third term, we have
\begin{align}
  \label{eq:69}
 x s_{\beta[r,1]} =  x s_{\beta[r-c,1]} s_{\beta[r,1]} s_{\beta[r+c,1]}
\end{align}
and we compute
\begin{align}
  \label{eq:70}
  x s_{\beta[r-c,1]} s_{\beta[r,1]} (\beta[r+c,1])
  &= x s_{\beta[r-c,1]} s_{\beta[r,1]} s_{\beta[r+c,1]} (-\beta[r+c,1] )\\
  &= x s_{\beta[r,1]}(-\beta[r+c,1] ) = x( \beta[r-c,1]) \notag
\end{align}
which is positive by construction.
\end{proof}

\noindent {\it Subsubcase:} $r=-1$, $\beta[r-1,1] \notin \Invplusplusx(s_{\beta[r,1]})$, and $\beta[-2,2] \in \Invplusplusx(s_{\beta[r,1]})$.

By the assumption that $\beta[-2,2] \in \Invplusplusx(s_{\beta[r,1]})$, we also have:
\begin{align}
  \label{eq:71}
  \beta[0,0] \in \Invplusplusx(s_{\beta[r,1]})
\end{align}
In this subsubcase, we claim that
\begin{align}
  \label{eq:72}
  x \rightarrow x s_{\beta[0,0]} \rightarrow x s_{\beta[0,0]} s_{\beta[1,-1]} \rightarrow x s_{\beta[-1,1]}
\end{align}
is a three-term chain.
\begin{proof}
  We have $x(\beta[0,0]) > 0$ by assumption.
  
For the second term, we calculate:
\begin{align}
  \label{eq:73}
  x s_{\beta[0,0]} (\beta[1,-1]) =  x(\beta[-1,1]) > 0.
\end{align}

For the third term, we have
\begin{align}
  \label{eq:75}
  x s_{\beta[0,0]} s_{\beta[1,-1]} s_{\beta[0,0]} = x s_{\beta[-1,1]}
\end{align}
and:
\begin{align}
  \label{eq:76}
  x s_{\beta[0,0]} s_{\beta[1,-1]} ( \beta[0,0]) = x s_{\beta[-1,1]}( - \beta[0,0]) = x ( \beta[-2,2]) > 0.
\end{align}
\end{proof}

\subsection{ {\bf Case 2:} $n > 0$ and $ \beta[r,n-1] \in \Invplusplusx(s_{\beta[r,n]})$}

In this case,
  \begin{align}
    \label{eq:49}
   x \rightarrow x s_{\beta[r,n-1]} \rightarrow x s_{\beta[r,n-1]}s_{\beta[r,-1]} \rightarrow x s_{\beta[r,n]}
  \end{align}
is a chain in the Bruhat order.
\begin{proof}
 Because $x(\beta[r,n-1])>0$, we have:
\begin{align}
  \label{eq:52}
   x \rightarrow x s_{\beta[r,n-1]}.
\end{align}

For the second term of the chain, we compute
\begin{align}
  \label{eq:53}
 x s_{\beta[r,n-1]}(\beta[r,-1]) =  x(\beta[r,2n-1)]) > 0
\end{align}
because $n \geq 1$ and Lemma \ref{lem:pos-prop-of-inv-plusplus}.

Finally, for the third term of the chain, we compute that
\begin{align}
  \label{eq:54}
  x s_{\beta[r,n-1]}s_{\beta[r,-1]}s_{\beta[r,0]} = x s_{\beta[r,n]}
\end{align}
and:
\begin{align}
  \label{eq:56}
  x s_{\beta[r,n-1]}s_{\beta[r,-1]}(\beta[r,0]) 
  &= x s_{\beta[r,n-1]}s_{\beta[r,-1]}s_{\beta[r,0]}( - \beta[r,0])\\
  &= xs_{\beta[r,n]}(- \beta[r,0]) = x(\beta[r,2n]) > 0.\notag
\end{align}
\end{proof}

\subsection{{\bf Case 3:} $n > 0$ and $ \beta[r,n-1]  \notin \Invplusplusx(s_{\beta[r,n]})$}
For this to occur, what fails is that:
\begin{align}
  \label{eq:41}
  x(\beta[r,n-1]) < 0
\end{align}
Equivalently, by \eqref{eq:3},  
\begin{align}
  \label{eq:40}
\sigma(r,n-1) \cdot \sigma(r-a,n-1-b+\ell(r-a)) < 0
\end{align}
However, since
\begin{align}
  \label{eq:42}
  x(\beta[r,n]) > 0
\end{align}
we have: 
\begin{align}
  \label{eq:43}
\sigma(r,n) \cdot \sigma(r-a,n-b+\ell(r-a)) > 0
\end{align}
By our assumption that $\sigma(r,n) > 0$, we have 
$\sigma(r-a,n-b+\ell(r-a)) = 1$.

Because Case 1 handles $\sigma(r,n-1)=-1$, we may assume $\sigma(r,n-1) = \sigma(r,n) = 1$. With this assumption, we have 
\begin{align}
  \label{eq:44}
  \sigma(r-a,n-1-b+\ell(r-a)) = -1
\end{align}
and
\begin{align}
  \label{eq:45}
 \sigma(r-a,n-b+\ell(r-a)) = +1 
\end{align}

\begin{Proposition}
  \label{prop:inv-on-slope-of-neg-ell}
Given the assumptions of this case,  $\Invplusplusx(s_{\beta[r,n]})$ lies on the line of slope $-\ell$ passing through $(r,n)$.
\end{Proposition}

\begin{proof}
\noindent{Case: $r- a = 0$.}
This case does not happen because:
\begin{align}
  \label{eq:106}
 \#\Invplusplusx(s_{\beta[r,n]}) = 1 . 
\end{align}

\noindent{Case: $r- a > 0$.}
In this case, we must have:
\begin{align}
  \label{eq:46}
 n-1-b+\ell(r-a) = -1. 
\end{align}
That is, the line through $(r,n)$ and $(a,b)$ has slope $-\ell$. Using the involution on $\Invplusplusx(s_{\beta[r,n]})$, we see that $\beta[\tr,\tn] \in \Invplusplusx(s_{\beta[r,n]})$ only if $(\tr,\tn)$ lies on this line passing through $(r,n)$ and $(a,b)$, which has slope $-\ell$. 

\noindent{Case: $r- a < 0$.}
In this case, we must have
\begin{align}
  \label{eq:47}
 n-1-b+\ell(r-a) = 0. 
\end{align}

First, let us consider when $b \geq 0$.
Therefore, $n= 1+b - \ell(r-a)$. Because $r-a <0$ and $b \geq 0$, we therefore have $n > \ell$. Hence $\beta[r+1, n- \ell] \in \Invplusplusx(s_{\beta[r,n]})$.

Let us now consider when $b < 0$. Suppose $\beta[\tr,\tn] \in \Invplusplusx(s_{\beta[r,n]})$. Using the involution on $\Invplusplusx(s_{\beta[r,n]})$, we may assume:
\begin{align}
  \label{eq:92}
  \tn - n + \ell(\tr -r) \leq 0
\end{align}
Because
\begin{align}
  \label{eq:93}
 n-1-b+\ell(r-a) = 0 
\end{align}
we have:
\begin{align}
  \label{eq:94}
   \tn  -1-b+\ell(r-a) + \ell(\tr -r) \leq 0.
\end{align}
Therefore:
\begin{align}
  \label{eq:95}
  \tn -1-b+ \ell(\tr -a) \leq 0.
\end{align}

Because
\begin{align}
  \label{eq:96}
  - \ell (\tr -a ) \geq  \tn -1 - b
\end{align}
and $\tn \geq 0$ and $b < 0$, we have
\begin{align}
  \label{eq:97}
  - \ell (\tr -a )  \geq 0
\end{align}
which implies:
\begin{align}
  \label{eq:98}
  \tr -a \geq 0
\end{align}

Requiring $x(\beta[\tr,\tn]) >0$ is equivalent to:
\begin{align}
  \label{eq:99}
  \sigma(\tr-a, \tn - b + \ell (\tr -a)) >0.
\end{align}

So if $\tr -a < 0$, we must have 
\begin{align}
  \label{eq:100}
  \tn - b + \ell (\tr -a) > 0
\end{align}
which implies (using \eqref{eq:94}) that 
\begin{align}
  \label{eq:101}
  \tn - b + \ell (\tr -a) = 1.
\end{align}

If $\tr -a = 0$, then we have to handle this case separately. Then we have 
\begin{align}
  \label{eq:102}
  0 \leq \tn - b \leq 1
\end{align}
which implies $\tn = 0$ and $b = 1$. We therefore still have:
\begin{align}
  \label{eq:103}
  \tn - b +\ell (\tr -a) = 1.
\end{align}
Using \eqref{eq:93}, we conclude
\begin{align}
  \label{eq:50}
  \tn + \ell \tr = n + \ell r
\end{align}
that is, $(\tr,\tn)$ lies on the line of slope $-\ell$ passing through $(r,n)$.
\end{proof}

\begin{Lemma}
  \label{lem:shift-by-ell}
 Suppose $\sigma(r,n) >0$ and $x(\beta[r,n]) > 0$. Suppose $d \geq 0$ and $\sigma(r+d,n-d\ell)>0$. Then $ x(\beta[r+d,n-d\ell]) > 0 $. 
\end{Lemma}
Let $c$ be the largest integer such that $x(\beta[r+c,n-c\ell]) > 0$; by Proposition \ref{prop:inv-on-slope-of-neg-ell} and \eqref{eq:105} we have $c \geq 1$. By Lemma \ref{lem:shift-by-ell}, $\beta[r+c,n-c\ell] \in \Invplusplusx(s_{\beta[r,n]})$. Then:
\begin{align}
  \label{eq:83}
x \rightarrow x s_{\beta[r+c, n- c\ell]} \rightarrow x s_{\beta[r,n]} s_{\beta[r+c, n- c\ell]}
\rightarrow x s_{\beta[r,n]}
\end{align}
is a three-term chain.

\begin{proof}
  
For the first term, we have $x(\beta[r+c, n- c\ell]) > 0$ by assumption.

For the second term, we calculate:
\begin{align}
  \label{eq:84}
 x s_{\beta[r,n]} s_{\beta[r+c, n- c\ell]} = x  s_{\beta[r+c, n- c\ell]} s_{\beta[r+c, n- c\ell]}   s_{\beta[r,n]} s_{\beta[r+c, n- c\ell]} 
\end{align}
Because of our conditions defining $c$, we have:
\begin{align}
  \label{eq:87}
  s_{\beta[r+c, n- c\ell]}( \beta[r,n]) = - \sigma(r,n) \sigma(r+2c,n-2c \ell)  \beta[r+2c,n-2c \ell] = \beta[r+2c,n-2c \ell].
\end{align}

Therefore
\begin{align}
  \label{eq:88}
x s_{\beta[r+c, n- c\ell]} \rightarrow x s_{\beta[r,n]} s_{\beta[r-c, n- c\ell]}
\end{align}
if and only if:
\begin{align}
  \label{eq:89}
x s_{\beta[r+c,n-c\ell]}(  \beta[r+2c,n-2c \ell]) > 0.
\end{align}
We have $ s_{\beta[r+c,n-c\ell]}(  \beta[r+2c,n-2c \ell]) > 0$, hence:
\begin{align}
  \label{eq:90}
  x s_{\beta[r+c,n-c\ell]}(  \beta[r+2c,n-2c \ell]) = x( \beta[r,n]) > 0.
\end{align}

For the third-term, we need:
\begin{align}
  \label{eq:91}
  x s_{\beta[r,n]} (\beta[r-c, n- c\ell]) < 0.
\end{align}
But this follows from $\beta[r-c,n-c\ell] \in \Invplusplusx(s_{\beta[r,n]})$.
\end{proof}

This completes the proof of Proposition \ref{prop:sl-two-cover-classification} and hence Theorem \ref{thm:ade-covering}.

\section{Further questions}
\label{sec:further-questions}

Although we have developed the Bruhat order on $\WTits$ and the length function in a combinatorial fashion, we expect both to have geometric and group-theoretic relevance. In this section we will describe some questions and conjectures about this perspective.

Let $\bG$ be an untwisted affine Kac-Moody group with positive and negative Borel subgroups $\bB$ and $\bB_-$, and let $k$ be a finite field. Likely we can relax these hypothesis to $\bG$ being general Kac-Moody and $k$ being an arbitrary field, but we retain these assumptions so that we can directly cite~\cite{BKP}. Let $F = k((\pi))$, the field of formal Laurent series over $k$, and let $\cO = k[[\pi]]$ be the ring of integers in $F$. Let $G = \bG(F)$, let $K = \bG(\cO)$, and let  $I = \{ g \in K \suchthat g \in \bB(k) \mod \pi \}$. Let $G^+ \subset \bG(F)$ be the {\bf Cartan semi-group}, i.e., the locus where the Cartan decomposition holds (see \cite{BKP,M} for the details). Furthermore we have a set-theoretic (not homomorphic) embedding $\WTits \subset G^+$ that is uniquely specified up to right multiplication by $I$. Then we have the following decomposition of $G^+$ (see \cite[Proposition 3.4.2 and Lemma 3.4.3]{BKP}).
\begin{Proposition}
We have an equality of subsets:
\begin{align}
  \label{eq:25}
G^+ = \bigsqcup_{x\in\WTits} I x I  
\end{align}
\end{Proposition}
\noindent Rephrasing this, the $I$-orbits on $G^+/I$ are exactly indexed by $\WTits$. 

\subsection{Double-affine Schubert cells}
If we momentarily consider the single-affine case of $\bG$ being finite type, then $G^+/I$ is precisely the $k$-points of the (single) affine flag variety, and the $I$ orbits on $G^+/I$ are precisely the (single) affine Schubert cells. And the (single) affine Bruhat order exactly describes the closure order on affine Schubert cells.

So following that, we will \emph{define} $G^+/I$ to be the {\bf $k$-points of the double-affine flag variety} and we will define the $I$ orbits on $G^+/I$ to be the {\bf double-affine Schubert cells}.
Following the single-affine heuristic, let us make the following definition. 

\begin{Definition}
 Let us define the {\bf closure} of a double-affine Schubert cell $Ix\cdot I/I$ by:
 \begin{align}
   \label{eq:26}
   \overline{Ix \cdot I/I} = \bigsqcup_{y \leq x} I y \cdot I/I
 \end{align}
 We will call sets of the form $\overline{Ix \cdot I/I}$ {\bf double-affine Schubert varieties}.
\end{Definition}

\begin{Question}
  {\label{ques:define-double-affine-flag-var}}
Can we define $G^+/I$ as an algebro-geometric object so that formula \eqref{eq:26} coincides with the closure in the Zariski topology?
\end{Question}

\subsection{Transverse slices}

One can easily see that $Ix \cdot I/I$ is an infinite set for most $x \in \WTits$, so there is no chance that $\overline{Ix \cdot I/I}$ is equal to the $k$-points of a finite-type $k$-scheme. Moreover, because the Bruhat order on $\WTits$ is unbounded below, it seems unlikely that $\overline{Ix \cdot I/I}$ is an ind-scheme of ind-finite type. Unfortunately, it seems that giving  $G^+/I$ a geometric structure will be comparably difficult to the problem of giving a geometric structure to the semi-infinite flag variety. So Question~\ref{ques:define-double-affine-flag-var}  may be too difficult.

However, what seems more reasonable is to work with transverse slices to a double-affine Schubert variety sitting inside one another. To define these objects (at the level of $k$-points), we need to introduce two other subgroups of $G$. Let $K_\infty = G( k[\pi^{-1}])$ and let $I_{\infty} = \{ g \in K_\infty \suchthat g \in \bB_-(k) \mod \pi^{-1} \}$. 

\begin{Definition}
Let $x,y \in \WTits$. Then let us define the transverse slice to $\overline{I y \cdot I/I}$ inside $\overline{I x \cdot I/I}$ as:
\begin{align}
  \label{eq:27}
  \overline{I x \cdot I/I} \cap I_\infty y \cdot I/I.
\end{align}
\end{Definition}

\begin{Conjecture}
  Let $x,y \in \WTits$. Then
  \begin{align}
    \label{eq:28}
     \overline{I x \cdot I/I} \cap I_\infty y \cdot I/I \neq \emptyset 
  \end{align}
  if and only if $x \leq y$.
\end{Conjecture}

We can also drop the closure, and make the following group theoretic conjecture 
\begin{Conjecture}
  Let $x,y \in \WTits$. Then
  \begin{align}
    \label{eq:28}
     I x \cdot I/I \cap I_\infty y \cdot I/I \neq \emptyset 
  \end{align}
  if and only if $x \leq y$.
\end{Conjecture}
A positive answer to this conjecture would give a purely group-theoretic definition of the Bruhat order without having to discuss closures.

\begin{Question}
Let $x \leq y$. Give the transverse slice 
$\overline{I x \cdot I/I} \cap I_\infty y \cdot I/I$ the structure of a finite-type affine scheme. 
\end{Question}

Following the situation in the single-affine case, we expect the transverse slice $\overline{I x \cdot I/I} \cap I_\infty y \cdot I/I$ to have dimension $\ell(y) - \ell(x)$. Unfortunately, we do not how to currently make that precise. However, we can make the following precise conjecture.

\begin{Conjecture}
 Let $x \leq y$. Then there exists a polynomial $R_{x,y} \in \ZZ[v]$ of degree $\ell(y) - \ell(x)$ independent of $k$ such that  
 \begin{align}
   \label{eq:29}
        \# (I x \cdot I/I \cap I_\infty y \cdot I/I ) = R_{x,y}(q)
 \end{align}
 where $q$ is the cardinality of $k$.
\end{Conjecture}

A positive answer to this would give a purely group-theoretic definition of the length function.

\subsection{$2$-dimensional phenomena}

Because an untwisted affine Kac-Moody group is itself constructed as a central extension of a loop group of a finite-type group, the $p$-adic group $G$ is a sort of double loop group. However, the two loops play very different roles in the discussion above. So a natural question is to understand $G^+/I$ from a purely $2$-dimensional point of view where the two loops play symmetric roles.

In the single-affine case, the loop group perspective gives rise to a well-understood relationship between the affine flag variety and spaces of bundles on an algebraic curve. Therefore, in the double-affine case, we expect there should be a relationship with bundles on an algebraic surface.

\begin{Question}
 Describe $G^+/I$ and/or the transverse slices  $I x \cdot I/I \cap I_\infty y \cdot I/I$ in terms of bundles on an algebraic surface. 
\end{Question}

If one considers the \emph{double-affine Grassmannian} $G^+/K$ instead of $G^+/I$, a candidate definition for transverse slices to $K$-orbit closures is given in terms of bundles on an algebraic surface by Braverman and Finkelberg in \cite{BF}. Even in that case, however, a precise bijection with the group-theoretic slice is unknown.

On a combinatorial level, this double loop phenemonena manifests itself in the fact that $\WTits$ contains two copies of the coroot lattice of finite-type group. The first copy lies in $W$, and the second copy arises because the Tits cone $\cT$ roughly looks like the coroot lattice times the semi-group of natural numbers. Therefore $2$-dimensional phenonemena from this point of view would be any non-trivial symmetry arising from interchanging these two lattices.

\subsection{Some combinatorial questions}

% \subsubsection{Poincar\'e series}

% Now that we have an integer-valued length function on $\WTits$, a natural question is to develop a theory of Poincar\'e series. However, na\"ive  Poincar\'e series
% \begin{align}
% \sum_{x\in \WTits} q^{\ell(x)}.
% \end{align}
% does not converge because there are clearly infinitely many $x$ with the same length.
% So the imprecise question is to develop a notion Poincar\'e series that does converge but still captures some non-trivial information about $\WTits$.

\subsubsection{Deodhar's inequality}

Recall that $\tDelta^+$ denotes the set of positive double-affine real roots. 
\begin{Conjecture}
Suppose $x,y,z \in \WTits$ with $x \leq  y \leq z$. Then we have the following inequality:
\begin{align}
  \label{eq:30}
 \# \{  \beta[n] \in \tDelta^+ \suchthat x \leq y s_{\beta[n]} \leq z \} \geq \ell(z) - \ell(x). 
\end{align}
\end{Conjecture}
In finite and single-affine cases, the above inequality is a conjecture of Deodhar that has since been proved by many authors. Although the statement is purely combinatorial, many of the proofs are intimately related to singularities of Schubert varieties and transverse slices. In our double-affine situation, we hope that a proof of this conjecture will shed some light on the geometry of transverse slices.

\subsubsection{Generalizing Coxeter group theory}

The theory of Coxeter groups and Bruhat orders is very rich. Although we are slowly developing analogues of many results for $\WTits$ and its Bruhat order, there are still many Coxeter-theoretic results that have not yet been generalized (see the book by Bj\"{o}rner and Brenti \cite{BB} for a nice exposition of many of these results). Below we list some problems that we think would be useful generalizations.
\begin{itemize}
\item Develop an analogue of reduced expressions and the subword criterion for the Bruhat order.
\item Develop weak order.
\item Develop a theory of Poincar\'e series.
\item Develop a notion of parabolic sub-semigroups.
\item Prove shellability results.
\item Classify short intervals. 
\end{itemize}

\end{document}